\documentclass[11pt]{amsart}
\usepackage{amssymb}
\usepackage{bm}
\usepackage{graphicx}

\newtheorem*{notn}{Notation}

\newtheorem{thm}{Theorem}[section]

\newtheorem{cor}[thm]{Corollary}
\newtheorem{Q}[thm]{Question}
\newtheorem{Def}[thm]{Definition}
\newtheorem{prop}[thm]{Proposition}
\newtheorem{rem}[thm]{Remark}

\newcommand{\bdfn}{\begin{Def} \rm}
\newcommand{\edfn}{\end{Def}}

\newcommand{\tfae}{the following are equivalent}

\newcommand{\ra}{\rightarrow}

\newcommand{\es}{\emptyset}
\newcommand{\ci}{\subseteq}

\newcommand{\al}{\alpha}
\newcommand{\be}{\beta}
\newcommand{\de}{\delta}
\newcommand{\e}{\varepsilon}

\newcommand{\La}{\Lambda}
\newcommand{\mb}{\mathbb}
\newcommand{\mc}{\mathcal}

\newcommand{\iy}{\infty}

\newcommand{\beqa}{\begin{eqnarray*}}
\newcommand{\eeqa}{\end{eqnarray*}}
\newcounter{cnt1}
\newcounter{cnt2}
\newcounter{cnt3}
\newcounter{cnt4}
\newcommand{\blr}{\begin{list}{$($\roman{cnt1}$)$} {\usecounter{cnt1}
\setlength{\topsep}{0pt} \setlength{\itemsep}{0pt}}}
\newcommand{\blR}{\begin{list}{\Roman{cnt4}.\ } {\usecounter{cnt4}
\setlength{\topsep}{0pt} \setlength{\itemsep}{0pt}}}
\newcommand{\bla}{\begin{list}{$(\alph{cnt2})$} {\usecounter{cnt2}
\setlength{\topsep}{0pt} \setlength{\itemsep}{0pt}}}
\newcommand{\bln}{\begin{list}{$($\arabic{cnt3}$)$} {\usecounter{cnt3}
\setlength{\topsep}{0pt} \setlength{\itemsep}{0pt}}}
\newcommand{\el}{\end{list}}

\begin{document}
\title[On property-$\bm{(R_1)}$ and relative Chebyshev centers]{On property-$\bm{(R_1)}$ and relative Chebyshev centers in Banach spaces-II}
\author[Das]{Syamantak Das}
\author[Paul]{Tanmoy Paul}
\address{Dept. of Mathematics\\
Indian Institute of Technology Hyderabad\\
India}
\email{ma20resch11006@iith.ac.in \&tanmoy@math.iith.ac.in}
\subjclass[2000]{Primary 41A28, 41A65 Secondary 46B20 41A50 \hfill \textbf{\today} }

\begin{abstract}
We continue to study (strong) property-$(R_1)$ in Banach spaces. As discussed by Pai \& Nowroji in [{\it On restricted centers of sets}, J. Approx. Theory, {\bf 66}(2), 170--189 (1991)], this study corresponds to a triplet $(X,V,\mc{F})$, where $X$ is a Banach space, $V$ is a closed convex set, and $\mc{F}$ is a subfamily of closed, bounded subsets of $X$. It is observed that if $X$ is a Lindenstrauss space then $(X,B_X,\mc{K}(X))$ has strong property-$(R_1)$, where $\mc{K}(X)$ represents the compact subsets of $X$. It is established that for any $F\in\mc{K}(X)$, $\textrm{Cent}_{B_X}(F)\neq\es$. This extends the well-known fact that a compact subset of a Lindenstrauss space $X$ admits a nonempty Chebyshev center in $X$. We extend our observation that $\textrm{Cent}_{B_X}$ is Lipschitz continuous in $\mc{K}(X)$ if $X$ is a Lindenstrauss space. If $Y$ is a subspace of a Banach space $X$ and $\mc{F}$ represents the set of all finite subsets of $B_X$ then we observe that $B_Y$ exhibits the condition for simultaneously strongly proximinal (viz. property-$(P_1)$) in $X$ for $F\in\mc{F}$ if $(X, Y, \mc{F}(X))$ satisfies strong property-$(R_1)$, where $\mc{F}(X)$ represents the set of all finite subsets of $X$. It is demonstrated that if $P$ is a bi-contractive projection in $\ell_\iy$, then $(\ell_\iy, Range (P), \mc{K}(\ell_\iy))$ exhibits the strong property-$(R_1)$, where $\mc{K}(\ell_\iy)$ represents the set of all compact subsets of $\ell_\iy$. Furthermore, stability results for these properties are derived in continuous function spaces, which are then studied for various sums in Banach spaces.
\end{abstract}
\maketitle
\section{Introduction}
\subsection{Prerequisites:~}
Some standard notations used in this study are introduced as follows: $X$ indicates a Banach
space, whereas a subspace denotes a closed linear subspace. For $x\in X$ and $r>0$,
$B(x,r)$ and $B[x,r]$ denote open and closed balls, respectively, each with its center at $x$ and radius $r$. Furthermore, $B_X$ and $S_X$ denote the closed unit ball and unit sphere of $X$,
respectively.
Further, $\mc{B}(X), \mc{C}(X), \mc{K}(X)$, and $\mc{F}(X)$ denote the set of all closed and
bounded, closed and convex, compact, and finite subsets of $X$, respectively. Real numbers are assumed to be the
underlying field for all spaces.
For $x\in X, V\in \mc{C}(X)$, and $B\in \mc{B}(X)$, the following are defined:
\begin{notn}\quad
\bln\item[\rm(1)] $r(x,B) = \sup\{\|x-b\|:b\in B\}$
\item[\rm(2)] $\emph{rad}_V(B) = \inf\{r(x,B):x\in V\}$
\item[\rm(3)] $\emph{Cent}_V(B) = \{v\in V:r(v,B)=\emph{rad}_V(B)\}$
\item[\rm(4)] $\de-\emph{Cent}_V(B) = \{v\in V:r(v,B)\leq \emph{rad}_V(B)+\de\}$
\item[\rm(5)] For $B\ci X$, $B_\eta=\{x\in X:d(x,B)\leq\eta\}$ \emph{for} $\eta>0$.
\item[\rm(6)] $S_\eta(B)=\{x\in X: r(x,B)\leq\eta\}$ \emph{for} $\eta>0$.
\el
\end{notn}
Note that \bln
\item $\textrm{Cent}_V(B)=\left\{\bigcap\limits_{b\in B}B[b,\textrm{rad}_V(B)]\right\}\cap
V$.
\item $\de-\textrm{Cent}_V(B)=\left\{\bigcap\limits_{b\in
B}B[b,\textrm{rad}_V(B)+\de]\right\}\cap V$.
\el
However, if $V\in \mc{C}(X)$ and $B\in\mc{B}(X)$, the set $\textrm{Cent}_V(B)$ may be
empty, although the set $\de-\textrm{Cent}_V(B)$ is nonempty for any $\de>0$.
A triplet $(X,V,\mc{F})$-- where $V\in \mc{C}(X)$ and $\mc{F}\ci \mc{B}(X)$, a subfamily of closed bounded subsets, exhibits the {\it restricted center property} ({\textrm{\bf{rcp}}}) if for all
$F\in\mc{F}, \textrm{Cent}_V(F)\neq\es$. Here, $\textrm{rad}_V(F)$ represents the
radius of the smallest ball (if it exists) in $X$ centered at $V$ and containing $F$,
$\textrm{Cent}_V(F)$ represents the possible points of the centers of these balls, and $\de-\textrm{Cent}_V(F)$ represents the set of points in $V$ which are at most $\textrm{rad}_V(F)+\de$ away from $F$.
Several researchers have investigated various characteristics (stated in the subsequent discussions) related to the entities defined above, viz. $\textrm{Cent}_V(F), \textrm{rad}_V(F)$ (see \cite{DP, LPST, PN, VL}), determined by various geometric properties of the Banach space and also the type of the closed convex subset $V$.

\bdfn\label{D2}
\cite{PN} Let $V\in \mc{C}(X)$ and $\mc{F}\ci \mc{B}(X)$. The triplet $(X,V,\mc{F})$ exhibits {\it property-$(R_1)$} if for
$v\in V, F\in\mc{F}$, and $r_1, r_2>0$, the conditions
$r(v,F)\leq r_1+r_2$ and $S_{r_2}(F)\cap V\neq\es$ imply that
$V\cap B[v,r_1+\e]\cap S_{r_2+\e}(F)\neq\es$, for all $\e>0$.
\edfn

Equivalently (see  \cite[Theorem~2.2, 2.4]{DP}), the triplet $(X,V,\mc{F})$ exhibits {\it property-$(R_1)$} if for $v\in V, F\in\mc{F}, r_1, r_2>0$ the conditions
$r(v,F)< r_1+r_2$ and $S_{r_2}(F)\cap V\neq\es$ imply that
$V\cap B[v,r_1]\cap S_{r_2}(F)\neq\es$. 

The above is a set-valued analogue of the $1\frac{1}{2}$-ball property (\cite{DY1}) and clearly a subspace $V$ has the $1\frac{1}{2}$-ball property in $X$ if $(X,V,\mc{F})$ has property-$(R_1)$ and $\mc{F}$ contains the singletons. The article by 
Pai and Nowroji (\cite{PN}) reported that if $(X,V,\mc{F})$ exhibits property-$(R_1)$, then it has ${\textrm {\bf{rcp}}}$. However, in \cite{IL} it is demonstrated that the $1\frac{1}{2}$-ball property is insufficient to ensure ${\bf \textrm{\bf {rcp}}}$ for finite subsets. In this context we recall \cite[Proposition~2.2]{TT}. It is observed for a family of bounded subsets $\mc{F}$, if for all $F\in\mc{F}$, $\textrm{Cent}_{B_Y}(F)\neq\es$  for a subspace $Y$, then for all $F\in\mc{F}$, $\textrm{Cent}_Y(F)\neq\es$. This concludes if $(X,B_X,\mc{F})$ exhibits ${\textrm {\bf{rcp}}}$, then $\textrm{Cent}_X(F)\neq\es$ for all $F\in\mc{F}$.

Several characterizations for property-$(R_1)$ have been derived by Daptari and Paul \cite{DP}. It is clear that for an arbitrary $V\in\mc{C}(X)$ and $v\in V$, $r(v,F)\leq\textrm{rad}_V(F)+d(v,\textrm{Cent}_V(F))$, for all $F\in\mc{B}(X)$.

\begin{thm}\cite[Theorem~2.4]{DP}
Let $V$ be a closed convex subset of $X$. Then, the triplet $(X,V,\mc{F})$ exhibits property-$(R_1)$ if and only if for $v\in V$ and 
$F\in\mc{F}$, $r(v,F)=\emph{rad}_V(F)+d(v,\emph{Cent}_V(F))$.
\end{thm}

Daptari and Paul \cite{DP1} studied a stronger version of property-$(R_1)$, called {\it strong property-$(R_1)$}, which is in fact a set-valued version of the {\it strong $1\frac{1}{2}$-ball property} (see~\cite{IL}).

\bdfn\label{D3}\cite{DP1} 
Let $V\in \mc{C}(X)$ and $\mc{F}\ci \mc{B}(X)$. The triplet $(X,V,\mc{F})$ exhibits {\it strong property-$(R_1)$} if for
$v\in V, F\in\mc{F}, r_1, r_2>0$ the conditions
$r(v,F)\leq r_1+r_2$ and $S_{r_2}(F)\cap V\neq\es$ imply that
$V\cap B[v,r_1]\cap S_{r_2}(F)\neq\es$.
\edfn

Several characterizations and examples of strong property-$(R_1)$ are provided in \cite{DP1}. Certain properties relevant to this study are listed below.

\begin{thm}\label{t01}
Let $X$ be a Banach space, $V$ be a subspace of $X$, and $\mc{F}$ be a subfamily of $\mc{B}(X)$. Then, \tfae.
\bla
\item $(X,V,\mc{F})$ exhibits strong property-$(R_1)$.
\item $(X,V,\mc{F})$ exhibits property-$(R_1)$ and $\forall~ F\in\mc{F}$, $\emph{Cent}_{B_V}(F)\neq\es$.
\item $\forall~ v\in V$ and $F\in\mc{F}$, $r(v,F)=\emph{rad}_V(F)+\|v-z\|$, for certain $z\in \emph{Cent}_V(F)$.
\el
\end{thm}
In all the above characterizations for (strong) property-$(R_1)$, one can choose $v=0$.
In \cite{DP} Daptari and Paul reported that the space $C(K)$ where $K$ is a compact Hausdorff space yields many subspaces that satisfy (strong) property-$(R_1)$. It is well-known that an $M$-ideal (see \cite[pg.1]{HW}) in a Lindenstrauss space is categorized as such a subspace (see \cite[Proposition~2.3]{PN} and \cite[Theorem~3.6]{DP1}). For instance, if $x^*$ is an extreme point of the dual unit ball of a Lindenstrauss space $X$, then $\ker (x^*)$ exhibits strong property-$(R_1)$ for the set of all compact subsets of $X$. The result in \cite[Proposition~ 2.3]{PN} follows directly in consideration of Theorem~\ref{T13}  when combined with  \cite[Theorem~ 3.5]{DP1}.
This study examined (strong) property-$(R_1)$ and establishes various consequences, stability properties, and examples thereof.

\subsection{Summary of results:} The remainder of this paper is structured as follows.

Section~2 discusses several phenomena associated with property-$(R_1)$ and strong property-$(R_1)$. It is demonstrated with respect to the finite subsets of a Lindenstrauss space, the unit ball exhibits strong property-$(R_1)$. Moreover, it is observed if $(X,Y,\mc{F}(X))$ exhibits strong property-$(R_1)$, then $(X,B_Y,\mc{F}(B_X))$ has property-$(P_1)$.

A projection $P:X\ra X$ such that $\|P\|\leq 1, \|I-P\|\leq 1$ is referred to as a bi-contractive projection on $X$.
In Section~3, the range of any bi-contractive projection in $\ell_\iy$ is derived as exhibiting strong property-$(R_1)$ with respect to the compact subsets. This concludes the unit ball of such subspaces exihibits restricted Chebyshev center for all compact subsets of $\ell_\iy$.

Section~4 demonstrates that the properties considered in this study remain stable under continuous function spaces. For a compact Hausdorff space $K$, $C(K,X)$ is considered to be the vector space of all continuous functions from $K$ that take values in $X$. For an $f\in C(K,X)$, $\sup_{k\in K}\|f(k)\|$ defines a norm that makes the space complete. It is demonstrated that if $(X,Y,\mc{F}(X))$ exhibits property-$(R_1)$, then $(C(K,X),C(K,Y),\mc{F}(C(K,X)))$ has property-$(R_1)$, and vice versa. This study adopts the technique used by Yost in \cite[Theorem~2.1]{DY1} to confirm that a Weierstrass--Stone subspace of $C(K,X)$ exhibits strong property-$(R_1)$ for the finite subsets. 

Section~5 discusses a few cases when (strong) property-$(R_1)$ is stable with respect to various sums of Banach spaces.

\section{Various aspects of property-$(R_1)$}

The following theorem can be obtained from certain standard inequalities: $|\textrm{rad}_V(F_1)-\textrm{rad}_V(F_2)|\leq d_H(F_1,F_2)$, $|r(v_1,F)-r(v_2,F)|\leq \|v_1-v_2\|$, for $v_1, v_2\in V$ and $F_1, F_2, F\in\mc{B}(X)$. Here, $d_H$ represents the Hausdorff metric defined over $\mc{B}(X)$.

\begin{thm}\label{T16}
Let $X$ be a Banach space and $V\in \mc{C}(X)$. If $(X,V,\mc{F}(X))$ exhibits property-$(R_1)$, then $(X,V,\mc{K}(X))$ exhibits property-$(R_1)$.
\end{thm}

Similar to Theorem~\ref{T16}, in certain cases, the strong property-$(R_1)$ of a $V\in \mc{C}(X)$ for finite subsets is sufficient for ensuring the same for compact subsets.

\begin{thm}\label{T11}
Let $X$ be a Banach space. Let $(X,\tau_X)$ represents a locally convex topological vector space, where $\tau_X$ be a topology weaker than $(X,\|.\|)$. In addition to that, we assume that any (norm) bounded net in $X$ has a $\tau_X$ convergent subnet in $X$, and $\|.\|$ is lower semicontinuous in $(X,\tau_X)$. Then, for a $\tau_X$-closed $V\in \mc{C}(X)$, $(X,V,\mc{K}(X))$ exhibits strong property-$(R_1)$ whenever $(X,V,\mc{F}(X))$ has strong property-$(R_1)$.
\end{thm}

\begin{proof}
Let $K\in\mc{K}(X)$. The aim is to prove that $r(0,K)=\textrm{rad}_V(K)+\|z\|$ for certain $z\in \textrm{Cent}_V(K)$.

There exists a sequence $F_n\in\mc{F}(X)$ such that $d_H(K,F_n)\ra 0$ as $n\ra\iy$. Thus,
$r(0,F_n)=\textrm{rad}_V(F_n)+\|z_n\|$,
where $z_n\in\textrm{Cent}_V(F_n)$ for all $n\in\mathbb{N}$. Because $(z_n)$ is bounded in $X$, there exists a subnet $(z_{n_i})$ of $(z_n)$ such that $z_{n_i}\ra z$ in $(X,\tau_X)$ for a certain $z\in V$. Hence,
\beqa
\|z\|\leq \liminf_i\|z_{n_i}\|=\liminf_i [r(0,F_{n_i})-\textrm{rad}_V(F_{n_i})]=r(0,K)-\textrm{rad}_V(K).
\eeqa

{\sc Claim:~} $r(z,K)=\textrm{rad}_V(K)$.

Because $z_{n_i}\ra z$ in $(X,\tau_X)$, the following is obtained.
\beqa
r(z,K) &\leq & \liminf_i r(z_{n_i},K)\\
       &\leq & \liminf_i [r(z_{n_i},F_{n_i})+d(F_{n_i},K)]\\
       &=& \liminf_i [\textrm{rad}_V(F_{n_i})+d(F_{n_i},K)]=\textrm{rad}_V(K)
\eeqa
Hence, $r(z,K)=\textrm{rad}_V(K)$, and thus $z\in\textrm{Cent}_V(K)$. Therefore, $d(0,\textrm{Cent}_V(K))\leq\|z\|\leq r(0,K)-\textrm{rad}_V(K)$. The other inequality is evident. This completes the proof.
\end{proof}

The following is obtained by applying Theorem~\ref{T11}.

\begin{cor}\label{C13}
Let $X$ be a dual (reflexive) Banach space and $V$ be a weak$^*$ (respectively, norm) closed convex subset of $X$. If $(X,V,\mc{F}(X))$ exhibits strong property-$(R_1)$, then $(X,V,\mc{K}(X))$ exhibits strong property-$(R_1)$.
\end{cor}

Following the arguments used in \cite[Lemma~ 2.1]{TT} it is easy to observe that,
\begin{thm}\label{T13}
	Let $Y$ be a subspace of $X$ and $\mc{F}$ be a subfamily of $\mc{B}(X)$. If $(X,B_Y,\mc{F})$ has (strong) property-$(R_1)$ then $(X,Y,\mc{F})$ has (strong) property-$(R_1)$.
\end{thm}

Let $V\in\mc{C}(X)$. Let us recall the following.
\bdfn\cite{PN}
For a subfamily $\mc{F}\ci \mc{B}(X)$, the triplet $(X,V,\mc{F})$ exhibits property-$(P_1)$ if for $\e>0$ and $F\in\mc{F}$ there exists $\de(\e,F)>0$ such that $\de-\textrm{Cent}_V(F)\ci \textrm{Cent}_V(F)+\e B_X$.
\edfn

$(X,V,\mc{F})$ exhibits property-$(P_1)$ if it has property-$(R_1)$. In addition, a Banach space $X$ is considered a Lindenstrauss space if $X^*$ is isometric with $L_1(\mu)$ for a certain measure space $(\Omega, \Sigma, \mu)$. \cite{JL1} is a standard reference for a comprehensive study of Lindenstrauss-type Banach spaces. Furthermore, $X$ is a Lindenstrauss space if and only if for any finite family of closed balls $\{B_i\}_{i=1}^n$ that are pairwise intersecting in $X$, they actually intersect in $X$ (\cite[Theorem~5.5, 6.1]{JL1}).
In this statement, $n$ can be replaced with $4$. As discussed in Section~ 1, our next theorem extends the fact that any finite subset of a Lindenstrauss space has a nonempty Chebyshev center.

\begin{thm}\label{T1}
Let $X$ be a Lindenstrauss space. Then, $(X, B_X, \mc{F}(X))$ has \emph{{\bf rcp}}.
\end{thm}
\begin{proof}
	Let $F\in \mc{F}(X)$ be such that $F=\{z_1,z_2,\ldots,z_n\}$. Consider the family of balls $\mathcal{B}=\left\{B_X, B[z_1,\textrm{rad}_{B_X}(F)], B[z_2,\textrm{rad}_{B_X}(F)],\ldots, B[z_n,\textrm{rad}_{B_X}(F)]\right\}$. Evidently, for any two $z_i, z_j\in F$, $\|z_i-z_j\|\leq 2. \textrm{rad}_{B_X}(F)$.
	
	Furthermore, for $z\in F$, $\|z-0\|\leq r(0,F)\leq r(s,F)+1$ for any $s\in B_X$. 
	
	Hence, $\|z-0\|\leq \textrm{rad}_{B_X}(F)+1$ from where it follows that $B[z,\textrm{rad}_{B_X}(F)]\cap B_X\neq\emptyset$. Because $X$ is $L_1$-predual, $\bigcap_{B\in\mathcal{B}}B\neq\emptyset$.
	
	Hence, the result follows.
\end{proof}

\begin{thm}\label{T2}
	Let $X$ be a Lindenstrauss space. Then, $(X,B_X,\mc{F}(X))$ exhibits property-$(P_1)$.
\end{thm}
\begin{proof}
Let $F\in\mc{F}(X)$ such that $F=\{z_1,\cdots,z_n\}$ and $\e>0$. Let $x\in\e-\textrm{Cent}_{B_X}(F)$. 

{\sc Claim:~} $x\in \textrm{Cent}_{B_X}(F)+\e B_X$.

Our assumption yields $\|x-z_i\|\leq\textrm{rad}_{B_X}(F)+\e$ for all $i=1\cdots,n$. Thereafter, $B[x,\e]\cap B[z_i,\textrm{rad}_{B_X}(F)]\neq\es$ for all $i=1\cdots,n$. Additionally, $\|z_i-z_j\|\leq2.\textrm{rad}_{B_X}(F)$. Further, we have $\|z_i-0\|\leq r(0,F)\leq r(s,F)+1$ for all $s\in B_X$ for all $i=1,\cdots,n$. Hence, $\|z_i-0\|\leq \textrm{rad}_{B_X}(F)+1$. Since $X$ is a Lindenstrauss space, we have $B[x,\e]\cap\left(\cap_{i=1}^nB[z_i,\textrm{rad}_{B_X}(F)]\right)\cap B_X\neq\es$. In other words, $B[x,\e]\cap \textrm{Cent}_{B_X}(F)\neq\es$, and hence the claim follows.
\end{proof}

\begin{rem}
\bla
\item Let us recall that (see~ \cite[Theorem~2.4]{DP}) the case $\de=\e$ in the definition of property-$(P_1)$ ensures property-$(R_1)$. Hence, we have $(X, B_X, \mc{F}(X))$ has property-$(R_1)$ when $X$ is a Lindenstrauss space. In fact, we can draw a stronger conclusion, as derived in Theorem~ \ref{T12}.
\item From \cite[Theorem~ 2.5]{DP}, it now follows that if $X$ is a Lindenstrauss space, then $\emph{Cent}_{B_X}:(\mc{F}(X),d_H)\to (\mc{B}(X),d_H)$ is Lipschitz continuous.
\el
\end{rem}

\begin{thm}\label{T12}
Let $X$ be a Lindenstrauss space. Then, $(X,B_X,\mc{F}(X))$ exhibits strong property-$(R_1)$.
\end{thm}

\begin{proof}
Let $x\in B_X$ and $F=\{z_1, z_2,\ldots,z_n\}\in \mc{F}(X)$. Furthermore, let $r(x,F)\leq r_1+r_2$ and $S_{r_2}(F)\cap B_X\neq\es$ for $r_1, r_2>0$.

The condition $r(x,F)\leq r_1+r_2$ ensures $B[x,r_1]\cap B[z_i,r_2]\neq\es$ for $1\leq i\leq n$. In addition, $S_{r_2}(F)\cap B_X\neq\es$ ensures that $\|z_i-z_j\|\leq 2r_2$ for $i\neq j$ and $B[z_i,r_2]\cap B_X\neq\es$. However, owing to $n.2.I.P.$ of $X$, $B_X\cap \bigcap_{i=1}^n B[z_i,r_2]\cap B[x,r_1]\neq\es$ is obtained. This completes the proof.
\end{proof}

As discussed in Theorem~\ref{t01}, $\textrm{Cent}_{B_Y}(F)\neq\es$ for $F\in\mc{F}(X)$ if $(X,Y,\mc{F}(X))$ exhibits strong property-$(R_1)$. In Theorem~\ref{T1} it is derived that $B_Y$ admits property-$(P_1)$ for a suitable  subfamily of finite subsets of $X$.

\begin{thm}\label{T1}
Let $(X,Y,\mc{F}(X))$ exhibit strong property-$(R_1)$. Then, $(X,B_Y,\mc{F}(B_X))$ has property-$(P_1)$.
\end{thm}

\begin{proof}
We utilize the techniques used in \cite[Theorem~2.9]{IL}. The detailed proof is included for completeness. 

Let $F\in \mc{F}(B_X)$, $d=\textrm{rad}_Y(F)>0$, and $r(0,F)<1+d$. Hence, there exists $\eta>0$ such that $r(0,F)-d=1-\eta$.

Considering\ $\varepsilon>0$, choose $0<\delta<1$ such that $\delta+\frac{3\delta}{\de+\eta}<\varepsilon$.

Now, $d=\textrm{rad}_{B_Y}(F)$ (see \cite[Theorem~ 3.8]{DP1}). Let $y\in\delta-\textrm{Cent}_{B_Y}(F)$(that is,$\ r(y,F)<\textrm{rad}_{B_Y}(F)+\delta)$.
Now, by strong property-$(R_1)$, $r(y,F)=\textrm{rad}_Y(F)+d(y,\textrm{Cent}_Y(F))=d+\underset{z\in\textrm{Cent}_Y(F)}{\inf}\|z-y\|$, that is, $d(y,\textrm{Cent}_Y(F))=r(y,F)-d<\delta$.
Hence, there exists $\ y_0\in \textrm{Cent}_Y(F)$ such that $\|y-y_0\|<\delta$. Clearly, $\|y_0\|<\|y\|+\delta\leq 1+\delta$.

{\sc Claim:} There exists $z\in \textrm{Cent}_Y(F)\cap B_Y$ such that $\|y-z\|<\varepsilon$.

Now, $r(0,F)-d=1-\eta=d(0,\textrm{Cent}_Y(F))$ and there exists $\ z_1\in\textrm{Cent}_Y(F)$ with $\|z_1\|=1-\eta$.

Let $w_\lambda=\lambda y_0+(1-\lambda)z_1$, then $\|w_\lambda\|\leq\lambda(1+\delta)+(1-\lambda)(1-\eta)=1+\delta\lambda-(1-\lambda)\eta$.

Now, $1+\delta\lambda-(1-\lambda)\eta=1 \iff 1-\lambda=\frac{\delta}{\delta+\eta} \iff \lambda=\frac{\eta}{\delta+\eta}$. 

Let $\lambda=\frac{\eta}{\delta+\eta}$ and $z=w_\lambda$, then $0<\lambda<1$.

$\|y_0-z\|=(1-\lambda)\|y_0-z_1\|\leq \frac{3\delta}{\delta+\eta}$ because $ \|y_0-z_1\|\leq 2+1=3$.

Additionally, $z\in \textrm{Cent}_Y(F)$ because $\textrm{Cent}_Y(F)$ is convex and $\|z\|\leq 1+\delta\lambda-(1-\lambda)\eta=1$.

Hence, we have $z\in \textrm{Cent}_{B_Y}(F)$ and $\|y-z\|\leq\|y-y_0\|+\|y_0-z\|<\delta+\frac{3\delta}{\delta+\eta}<\varepsilon$.
This proves the claim.

Hence, $\delta-\textrm{Cent}_{B_Y}(F)\subseteq \textrm{Cent}_{B_Y}(F)+\varepsilon B_X$.
\end{proof}

\begin{thm}
Let $Y$ be a subspace of a Banach space $X$. If $(X^{**}, Y^{\perp\perp}, \mc{F}(X^{**}))$ has property-$(R_1)$, then $(X,Y,\mc{F}(X))$ has property-$(R_1)$.
\end{thm}
\begin{proof}
The proof is straightforward and follows from the extended version of {\it Principle of local reflexivity}, as stated in \cite[Theorem~3.2]{EB}.
\end{proof}

We do not have the answer to the following question.
\begin{Q}
Let $Y$ be a subspace of $X$. Does the triplet $(X^{**},Y^{\perp\perp},\mc{F}(X^{**}))$ have property-$(R_1)$ if $(X,Y,\mc{F}(X))$ has the property-$(R_1)$?
\end{Q}

\section{Examples from subspaces of $\ell_\iy$} 
Let us recall that a subspace $Y$ of $X$ has the $1\frac{1}{2}$-ball property if and only if for $x\in X$, $\|x\|=d(x,Y)+d(0, P_Y(x))$ (see~ \cite{GG}). Here  $P_Y(x)=\{y\in Y:\|x-y\|=d(x,Y)\}$, the set of points in $Y$ which are nearest to $x$. 
It is well-known that in $\ell_\iy (2)$, the subspace $span \{(1,1)\}=Z$ (say) has the $1\frac{1}{2}$-ball property.
The simplest manner of observing this is that for any $(p,q)\in \ell_\iy (2)$, $d((p,q), Z)=\frac{|p-q|}{2}$ and the unique best approximation from $(p,q)$ to the subspace $Z$ is $(\frac{p+q}{2},\frac{p+q}{2})$. Hence for $(p,q)\in\ell_\iy (2)$, the above identity turns out to be equivalent to $|p|\vee |q|=|\frac{p-q}{2}|+|\frac{p+q}{2}|$, which is true for an arbitrary pair of real numbers $p,q$.

Consequently, using similar arguments, it can be easily observed that $span\{(1,-1)\}$ (=$W$ say) has the $1\frac{1}{2}$-ball property in $\ell_\iy (2)$. This study claims that both subspaces, viz. $Z, W$, exhibit property-$(R_1)$ for finite subsets of $\ell_\iy (2)$.

\begin{prop}\label{p9}
Let $X=\ell_\iy(2)$ and $Z=\textrm{span}\{(1,1)\}$. Then, $(X,Z,\mc{F}(X))$ exhibits strong property-$(R_1)$.
\end{prop}

\begin{proof}
Let $F\in\mc{F}(X)$. Furthermore, let $(z,z)\in Z$ and $r_1,r_2>0$ be such that $r((z,z),F)\leq r_1+r_2$ and $S_{r_2}(F)\cap Z\neq\es$.

{\sc Claim:} $B[(z,z),r_1]\cap S_{r_2}(F)\cap Z\neq\es$.

If $card (F)=1$, then the claim follows from the fact that $Z$ has the $1\frac{1}{2}$-ball property in $X$.
Suppose that the assertion holds for $card (F)=n$. The following proof is when $card (F)=n+1$. Assume that $F=\{(x_1,y_1),(x_2,y_2),\ldots, (x_{n+1},y_{n+1})\}$.

Let $(p,p)\in S_{r_2}(F)\cap Z=\cap_{i=1}^{n+1}B[(x_i,y_i),r_2]\cap Z$. Because $Z$ exhibits strong property-$(R_1)$ when $F$ has $n$ elements, the following is obtained for certain $(s_i,s_i)\in\ell_\iy(2)$ for $i=1,2,\ldots,n$: $(s_i,s_i)\in B[(z,z),r_1]\cap \cap_{\substack{{j=1}\\{j\neq i}}}^{n+1}B[(x_j,y_j),r_2]$. Here, $s_1, s_2$ are chosen and the arguments are as stated below.

{\sc Case 1:~} When $p\leq s_1\leq s_2$.

Then, $-r_2\leq p-x_1\leq s_1-x_1\leq s_2-x_1\leq r_2$ and $-r_2\leq p-y_1\leq s_1-y_1\leq s_2-y_1\leq r_2$, and thus $(s_1,s_1)\in B[(x_1,y_1),r_2]$.

Thus, $(s_1,s_1)\in B[(z,z),r_1]\cap\bigcap_{i=1}^{n+1}B[(x_i,y_i),r_2]\cap Z=B[(z,z),r_1]\cap S_{r_2}(F)\cap Z$.

Similar arguments can be used for the cases:

{\sc Case 2:~}$p\leq s_2\leq s_1$, 

{\sc Case 3:~}$s_1\leq s_2\leq p$ and 

{\sc Case 4:~}$s_2\leq s_1\leq p$. 

We now remain with the following cases.

{\sc Case 5:~} When $s_1\leq p\leq s_2$.

Then, $z\leq s_1+r_1\leq p+r_1$ and $p-r_1\leq s_2-r_1\leq z$, and thus $|p-z|\leq r_1$.

Thus, $(p,p)\in B[(z,z),r_1]\cap\bigcap_{i=1}^{n+1}B[(x_i,y_i),r_2]\cap Z=B[(z,z),r_1]\cap S_{r_2}(F)\cap Z$.

{\sc Case 6:~} When $s_2\leq p\leq s_1$.

Then, $z\leq s_2+r_1\leq p+r_1$ and $p-r_1\leq s_1-r_1\leq z$, and thus $|p-z|\leq r_1$.

Thus, $(p,p)\in B[(z,z),r_1]\cap\bigcap_{i=1}^{n+1}B[(x_i,y_i),r_2]\cap Z=B[(z,z),r_1]\cap S_{r_2}(F)\cap Z$.

Hence, the result follows when $card(F)=n+1$, and this completes the proof.
\end{proof}

Moreover, a similar conclusion can be derived for the subspace $span \{(1,-1)\}$.

\begin{prop}\label{p10}
Let $X=\ell_\iy(2)$ and $W=\textrm{span}\{(1,-1)\}$. Then, $(X,W,\mc{F}(X))$ has strong property-$(R_1)$ in $X$.
\end{prop}

We now establish that the range of any bi-contractive projection in $\ell_\iy$ exhibits property-$(R_1)$.  Let us recall that, if $P:\ell_\iy\ra\ell_\iy$ is a bi-contractive projection, then either $Px(n)=\frac{x(n)+x(\tau (n))}{2}$ or $Px(n)=\frac{x(n)-x(\tau (n))}{2}$ for $x\in\ell_\iy$ and $n\in\mb{N}$ (see \cite[Theorem~3.9]{BR}). Here, $\tau:\mb{N}\ra\mb{N}$ is a permutation such that $\tau^2=I$ (identity).

Thus, if $P$ is a bi-contractive projection on $\ell_\iy$ and $\tau$ is the corresponding permutation on $\mb{N}$, then, for $n\in\mb{N}$, either $\tau(n)=n$ or $\tau(n)=m$ for certain $m\in\mb{N}$ and, in the second case, $\tau (m)=n$. This concludes for a bi-contractive projection $P:\ell_\iy\ra\ell_\iy$,
\bln
\item When $\tau(n)=n$: either $Px(n)=x(n)$ or $Px(n)=0$.
\item When $\tau(n)=m$ ($n\neq m$): either $(Px(n),Px(m))=\al (1,1)$, or $(Px(n),Px(m))=\al (1,-1)$ for some scalar $\al$.
\el
Hence, if $P\neq 0$, then $P(\ell_\iy)$ is isometrically isomorphic with either $\left(\Pi_{n\in A} (\al_n,\al_n)\right)_\iy\oplus_\iy\left(\Pi_{m\in B} (\be_m,-\be_m)\right)_\iy\oplus_\iy W$ or $\ell_\iy$. Here, $W$ is an $M$-summand of $\ell_\iy$ and $A, B\ci \mb{N}$. Based on \cite[Theorem~3.5]{DP1}, $W$ exhibits strong property-$(R_1)$ in $\ell_\iy$. Furthermore, the subspaces $\left(\Pi_{n\in A} (\al_n,\al_n)\right)_\iy$ and $\left(\Pi_{m\in B} (\be_m,-\be_m)\right)_\iy$ exhibit strong property-$(R_1)$ in $\ell_\iy$ when considering Propositions~\ref{p9} and \ref{p10} and Theorem~\ref{T5}. Finally, according to Theorem~\ref{T5}, the subspace $range (P)$ has strong property-$(R_1)$ in $\ell_\iy$. Hence, the following is obtained.

\begin{thm}\label{T9}
Let $P$ be a bi-contractive projection in $\ell_\iy$ and $Y=range(P)$. Then, $(\ell_\iy,Y,\mc{F}(\ell_\iy))$ has strong property-$(R_1)$.
\end{thm}

We now consider the derivation for the subspace $span \{(1,1,\ldots,1,\ldots)\}$, which exhibits strong property-$(R_1)$ in $\ell_\iy$ for finite subsets.

\begin{thm}\label{T3}
Let $X=\ell_\iy$ and $Y=\textrm{span}\{(1,1,\ldots)\}$. Then, $(X,Y,\mc{F}(X))$ exhibits strong property-$(R_1)$.
\end{thm}

\begin{proof}
Let $F\in\mc{F}(X)$, $(x,x,\ldots)\in Y$, $r_1,r_2>0$ be such that $r((x,x,\ldots),F)\leq r_1+r_2$ and $S_{r_2}(F)\cap Y\neq\es$. 

{\sc Claim:} $B[(x,x,\ldots),r_1]\cap S_{r_2}(F)\cap Y\neq\es$.

{\large{\sc Step 1:}} Let $card (F)=1$. Let $F=\{(x(1),x(2),\ldots)\}$. Then, we have $B[(x,x,\ldots),r_1]\cap B[(x(1),x(2),\ldots),r_2]\neq\es$ and $B[(x(1),x(2),\ldots),r_2]\cap Y\neq\es$. Let $(z(1),z(2),\ldots)\in B[(x,x,\ldots),r_1]\cap B[(x(1),x(2),\ldots),r_2]$ and $(y,y,\ldots)\in B[(x(1),x(2),\ldots),r_2]$. 

Let $\al=\underset{i\in\mathbb{N}}{\inf}z(i)$ and $\be=\underset{i\in\mathbb{N}}{\sup}z(i)$. Then, $\al\leq z(i)\leq\be$ for all $i\in\mathbb{N}$. 

{\sc Case 1:} When $y\leq \al$.

Then, $y\leq \al\leq z(i)$ for all $i\in\mathbb{N}$. Then, $-r_2\leq y-x(i)\leq \al-x(i)\leq z(i)-x(i)\leq r_2$ for all $i\in\mathbb{N}$. Hence, $|\al-x(i)|\leq r_2$ for all $i\in\mathbb{N}$. Let $\e>0$. Then, there exists $N\in\mathbb{N}$ such that $z(N)-\e\leq\al$. Hence, $x-\e\leq z(N)+r_1-\e\leq\al+r_1$ and $\al-r_1\leq z(N)-r_1\leq x$. Thus, $|\al-x|\leq r_1+\e$. 

Therefore, $(\al,\al,\ldots)\in B[(x,x,\ldots),r_1]\cap B[(x(1),x(2),\ldots),r_2]$.

{\sc Case 2:} When $\al\leq y\leq\be$.

Let $\e>0$. There exists $N,N'\in\mathbb{N}$ such that $z(N)-\e\leq\al$ and $\be\leq z(N')+\e$. Subsequently, $x-\e\leq z(N)-\e+r_1\leq\al+r_1\leq y+r_1$ and $y-r_1\leq\be-r_1\leq z(N')+\e-r_1\leq x+\e$. Then, $|y-x|\leq r_1+\e$.

Hence, $(y,y,\ldots)\in B[(x,x,\ldots),r_1]\cap B[(x(1),x(2),\ldots),r_2]$.

{\sc Case 3:} When $\be\leq y$.

Then, $z(i)\leq \be\leq y$ for all $i\in\mathbb{N}$. Furthermore, $-r_2\leq z(i)-x(i)\leq \be-x(i)\leq y-x(i)\leq r_2$ for all $i\in\mathbb{N}$. Hence, $|\be-x(i)|\leq r_2$ for all $i\in\mathbb{N}$. Let $\e>0$. Then, there exists $N\in\mathbb{N}$ such that $\be\leq z(N)+\e$. Hence, $x\leq z(N)+r_1\leq\be+r_1$ and $\be-r_1\leq z(N)-r_1+\e\leq x+\e$. Thus, $|\be-x|\leq r_1+\e$.

Thus, $(\be,\be,\ldots)\in B[(x,x,\ldots),r_1]\cap B[(x(1),x(2),\ldots),r_2]$.

Hence, $(X,Y,\mc{F}(X))$ has strong property-$(R_1)$ when $card (F)=1$.

{\large{\sc Step 2:}} Suppose that the assertion holds for all $F\in\mc{F}(X)$ when $card (F)=n$. Let $card (F)=n+1$. Let $F=\{x_1,\ldots,x_{n+1}\}$, where $x_i=(x_i(1),x_i(2),\ldots)$ for $i=1,\ldots,n+1$. Then, we have $B[(x,x,\ldots),r_1]\cap B[(x_i(1),x_i(2),\ldots),r_2]\neq\es$ for all $i=1,\ldots,n+1$ and $\cap_{i=1}^{n+1}B[(x_i(1),x_i(2),\ldots),r_2]\cap Y\neq\es$. Now, because $(X,Y,\mc{F}(X))$ exhibits property-$(R_1)$ when $card (F)=n$, the following is obtained: $B[(x,x,\ldots),r_1]\cap\cap_{\substack{{i=1}\\{i\neq 2}}}^{n+1}B[(x_i(1),x_i(2),\ldots),r_2]\cap Y\neq\es$ and $B[(x,x,\ldots),r_1]\cap \cap_{i=2}^{n+1}B[(x_i(1),x_i(2),\ldots),r_2]\cap Y\neq\es$. Let $(p,p,\ldots)\in B[(x,x,\ldots),r_1]\cap\cap_{\substack{{i=1}\\{i\neq 2}}}^{n+1}B[(x_i(1),x_i(2)\ldots),r_2]$, $(q,q,\ldots)\in B[(x,x,\ldots),r_1]\cap \cap_{i=2}^{n+1}B[(x_i(1),x_i(2),\ldots),r_2]$ and $(s,s,\ldots)\in \cap_{i=1}^{n+1}B[(x_i(1),x_i(2),\ldots),r_2]$. 

{\sc Case 1:} When $s\leq p\leq q$.

Then, $-r_2\leq s-x_2(i)\leq p-x_2(i)\leq q-x_2(i)\leq r_2$ for all $i\in\mathbb{N}$. Thus, $(p,p,\ldots)\in B[(x_2(1),x_2(2),\ldots),r_2]$.

Hence, $(p,p,\ldots)\in B[(x,x,\ldots),r_1]\cap \cap_{i=1}^{n+1}B[(x_i(1),x_i(2),\ldots),r_2]$.

Similar ideas can be adopted to establish the following cases.

{\sc Case 2:} $s\leq q\leq p$.

{\sc Case 3:} $p\leq q\leq s$.

{\sc Case 4:} $q\leq p\leq s$.

We now remain with the following cases.

{\sc Case 5:} $p\leq s\leq q$.

Then, $x\leq p+r_1\leq s+r_1$ and $s-r_1\leq q-r_1\leq x$. Thus, $|s-x|\leq r_1$. 

Hence, $(s,s,\ldots)\in B[(x,x,\ldots),r_1]\cap \cap_{i=1}^{n+1}B[(x_i(1),x_i(2),\ldots),r_2]$.

{\sc Case 6:} When $q\leq s\leq p$.

Then, $x\leq q+r_1\leq s+r_1$ and $s-r_1\leq p-r_1\leq x$. Thus, $|s-x|\leq r_1$. 

Hence, $(s,s,\ldots)\in B[(x,x,\ldots),r_1]\cap \cap_{i=1}^{n+1}B[(x_i(1),x_i(2),\ldots),r_2]$.

Thus, the assertion holds for all $F\in\mc{F}(X)$.
\end{proof}

It is clear that the subspaces of type $\left(\Pi_n (\al_n,\al_n)\right)_\iy$ or $\left(\Pi_m (\be_m,-\be_m)\right)_\iy$ stated before Theorem~\ref{T9} are $w^*$-closed, and so is the subspace in Theorem~\ref{T3}. Hence, by Corollary~\ref{C13}, the conclusions in Theorems~\ref{T9} and \ref{T3} remain valid for $\mc{K}(X)$.

\section{Subspaces of $C(K,X)$ with property-$(R_1)$}
In this section, the following fact is proven. By $K$ and $C(K,X)$, we denote a compact Hausdorff space and the Banach space of $X$-valued continuous functions over $K$, as discussed in Section~$1$.

\begin{thm}\label{T4}
Let $Y$ be a subspace of $X$. Then, $(X, Y, \mc{F}(X))$ has property-$(R_1)$ if and only if $(C(K,X), C(K,Y), \mc{F}(C(K,X)))$ has property-$(R_1)$.
\end{thm}

Before proving Theorem~\ref{T4}, a few supporting results must be derived. For a real valued function $f:S\to\mb{R}$, we denote $S(f)=\overline{\{t\in S:f(t)\neq 0\}}$, the support of $f$.

\begin{prop}\label{P1}
Let $f_1,f_2,\ldots, f_k\in C(K)$ and $\e>0$. Then, there exists a finite family $(\varphi_i)_{i=1}^m\ci C(K)$, where $(\varphi_i)_{i=1}^m$ forms a partition of unity and there exists $h_1,h_2,\ldots, h_k\in span\{\varphi_i:1\leq i\leq m\}$ such that $\|f_i-h_i\|_\iy<\e$ for $1\leq i\leq k$.
\end{prop}
\begin{proof}
{\sc Case 1:} When $k=1$.

Let $\{V_i:1\leq i\leq n\}$ be a finite open cover of $K$ such that $|f(z)-f(w)|<\e$ for $z, w\in V_i$ and $1\leq i\leq n$. Let $(\varphi_i)_{i=1}^n$ be a partition of unity such that $0\leq \varphi_i\leq 1$, $1\leq i\leq n$ and $S(\varphi_i)\ci V_i$ (\cite[Theorem~2.13]{WR}). Choose $v_i\in V_i$ and define $h=\sum_{i=1}^nf(v_i)\varphi_i$.
Then, $|f(x)-h(x)|=|\sum_{i=1}^nf(x)\varphi_i(x)-\sum_{i=1}^nf(v_i)\varphi_i(x)|\leq\sum_j|\varphi_{i_j}(x)||f(x)-f(v_{i_j})|$. The last sum is taken over all those $j$'s for which $x\in V_{i_j}$. Clearly, $\sum_j\varphi_{i_j}(x)|f(x)-f(v_{i_j})|\leq \e$.

{\sc Case 2:} When $k>1$.

Without loss of generality, this study assumes that $k=2$; no new ideas are involved in other values of $k$.

Let $(\varphi_i)_{i=1}^n$ and $(\varrho_i)_{i=1}^m$ be two partitions of unity in $C(K)$, such that there exists $h_1, h_2$, where $h_1\in span\{\varphi_i:1\leq i\leq n\}$ and $h_2\in span\{\varrho_i:1\leq i\leq m\}$, where $\|f_i-h_i\|<\e$ for $i=1, 2$. Then, $\{\varphi_i\varrho_j:1\leq i\leq n, 1\leq j\leq m\}$ is a partition of unity and $h_i\in span\{\varphi_i\varrho_j:1\leq i\leq n, 1\leq j\leq m\}$. This completes the proof.
\end{proof}

Let $(\varphi_i)_{i=1}^N$ be a finite partition of unity in $C(K)$ corresponding to an open cover $\mc{U}$ of $K$ obtained as in \cite[Theorem~2.13]{WR}. Then, we call $(\varphi_i)_{i=1}^N$ a partition of unity in $C(K)$ subordinate to the cover $\mc{U}$. In this case, $(\varphi_i)_{i=1}^N$ corresponds to a subspace $Z$ of $C(K,X)$: $Z=\{\sum_{i=1}^Nx_i\varphi_i: x_i\in X, 1\leq i\leq n\}$.
It is clear that $Z\cong \bigoplus_{\ell_\iy(N)} X$.

\begin{prop}\label{P2}
Let $X$ be a Banach space, $K$ be a compact Hausdorff space, and $f_1, f_2,\ldots, f_n\in C(K,X)$. Then, for $\e>0$, there exists a subspace $Z$ of $C(K,X)$, where $Z\cong \bigoplus_{\ell_\iy(m)} X$ for some $m$, and $d(f_i, Z)\leq\e$, $1\leq i\leq n$.
\end{prop}
\begin{proof}
This study followed Proposition~\ref{P1} to construct a subspace $Z\ci C(K,X)$. If $S_i=f_i(K)$, then $S_i\ci X$ is a compact set. Let us fix $i$.
For every $s\in S_i$, let $B(s,\e)\cap S_i$ be a ball in $S_i$. For a finite sub-cover $\mc{U}_i$ of $\{f_i^{-1}(B(s,\e)\cap S_i):s\in S_i\}$, we may choose a finite partition of unity subordinate to the cover $\mc{U}_i$; say, $(\varphi_j)$. If $(s_j)_{j=1}^n\ci S_i$ is a finite set of points corresponding to the cover $\mc{U}_i$, then $\|f_i-\sum_j \varphi_js_j\|\leq\e$.

We now continue the process for other values of $i$. Following the arguments for $n$ functions as derived in Proposition~\ref{P1}, it is concluded for a finite dimensional subspace $Z$ of $C(K,X)$, where $Z\cong\bigoplus_{\ell_\iy(m)} X$, for some $m\in\mb{N}$.
Clearly, $d(f_i, Z)\leq\e$, $1\leq i\leq n$, and this completes the proof.
\end{proof}

We state the following result without proof, which is useful to derive Theorem~\ref{T10}. A routine verification of the (strong) property-$(R_1)$ can lead to the proof of the following. We derive similar results in Section~5. For any unexplained notation used in Theorem~\ref{T6}, we refer to section~5.

\begin{thm}\label{T6}
Let $Y$ be a subspace of $X$. Then, $(X,Y,\mc{B}(X))$ exhibits (strong) property-$(R_1)$ if and only if $(X_\iy, Y_\iy,\mc{B} (X_\iy))$ exhibits (strong) property-$(R_1)$.
\end{thm}

We are now ready to prove Theorem~\ref{T4}.
\begin{thm}\label{T10}
Let $Y$ be a subspace of $X$. Then, $(X, Y, \mc{F}(X))$ has property-$(R_1)$ if and only if $(C(K,X), C(K,Y), \mc{F}(C(K,X)))$ has property-$(R_1)$.
\end{thm}

\begin{proof}
Let $F\in \mc{F}(C(K,X))$ and $g\in C(K,Y)$. Let $r_1, r_2>0$ such that $S_{r_2}(F)\cap C(K,Y)\neq\es$ and $r(g,F)\leq r_1+r_2$. Consequently, the following is developed.

{\sc Claim:} $S_{r_2+\e}(F)\cap B(g,r_1+\e)\cap C(K,Y)\neq\es$ for all $\e>0$.

Suppose that $h\in S_{r_2}(F)\cap C(K,Y)$.
From Proposition~\ref{P2}, there exists $Z\ci C(K,X)$, where $Z\cong \bigoplus_{\ell_\iy(k)} X$ for certain $k$ such that for all $f\in F$, $d(f,Z)<\e$. Let $F=\{f_1,f_2,\ldots,f_n\}$ and $F^\prime =\{f_1^\prime,f_2^\prime,\ldots,f_n^\prime\}$ where $f_i^\prime\in Z$ and $\|f_i-f_i^\prime\|<\e$. Additionally, there exists $W\ci C(K,Y)$, where $W\cong\bigoplus_{\ell_\iy(m)}Y$ such that there exist $g^\prime, h^\prime\in W$ and $\|g-g^\prime\|<\e, \|h-h^\prime\|<\e$, here $g$ and $h$ are taken as above. Subsequently, without loss of generality, $m=k$ may be assumed. Furthermore, from the assumption, $S_{r_2+2\e}(F^\prime)\cap W\neq\es$ and $r(g^\prime, F^\prime)\leq r_1+r_2+2\e$ is obtained. In addition, from Theorem~\ref{T6} $(\bigoplus_{\ell_\iy(k)}X, \bigoplus_{\ell_\iy(k)}Y, \mc{F}(\bigoplus_{\ell_\iy(k)}X))$ with property-$(R_1)$ is obtained; hence,
$S_{r_2+2\e}(F^\prime)\cap B(g^\prime,r_1+2\e)\cap W\neq\es$.

Let $h\in S_{r_2+2\e}(F^\prime)\cap B(g^\prime,r_1+2\e)\cap W$. After identifying 
$h$ with an element in $C(K,Y)$, we obtain $h\in S_{r_2+3\e}(F)\cap B(g,r_1+3\e)\cap C(K,Y)$. Moreover, because $\e>0$ is arbitrary, the proof follows.
\end{proof}

However, it is not yet known whether analogous results such as Theorem~\ref{T10} are true for the spaces of the form $L_1(\mu, X)$ or not. Nevertheless, if the triplet $(L_1(\mu,X),L_1(\mu,Y),\mc{F}(L_1(\mu,X)))$ has property-$(R_1)$, then $(X,Y,\mc{F}(X))$ has property-$(R_1)$.

\begin{thm}\label{1}
Let $E$ be a real Lindenstrauss space, $K$ and $S$ be compact Hausdorff spaces, and $\psi:K\rightarrow S$ be a continuous onto map. Let $\psi^* : C(S, E) \rightarrow C(K,E)$ be the natural isometric embedding expressed as $\psi^*f = f\circ\psi$. Then, $(C(K,E),\psi^*C(S,E), \mc{F}(C(K,E)))$ exhibits strong property-$(R_1)$.
\end{thm}
\begin{proof}
  Let $F\in \mc{F}(C(K,E))$ such that $F=\{f_1,\ldots,f_n\}$. Suppose that $r>0$ such that $r(0,F)\leq1+r$ and $S_r(F)\cap \psi^*C(S,E)\neq\es$. Furthermore, define $\eta:S\rightarrow \mc{B}(E)$ by
  \beqa
      \eta(y) &=& B[0,1]\cap\left(\underset{k\in\psi^{-1}(y)}{\bigcap}\cap^n_{i=1} B[f_i(k),r]\right)\\
      &=& B[0,1]\cap\left(\cap^n_{i=1}\{a\in E: f_i(\psi^{-1}(y))\subseteq B[a,r]\}\right).
  \eeqa
  Each $\eta(y)$ is closed and convex.
  
 {\sc Claim:} $\eta(y)\neq\es$ for all $y\in S$.
  
  Let $\psi^*g\in \psi^*C(S,E)\cap S_r(F)=\psi^*C(S,E)\cap \left(\cap^n_{i=1} B[f_i,r]\right)$. 
  
  \mbox{For~} $k_1,k_2\in \psi^{-1}(y)$, we have
  \beqa
  \|f_i(k_1)-f_j(k_2)\| &\leq &\|f_i(k_1)-g(y)\|+\|f_j(k_2)-g(y)\|\\
     &\leq&\|f_i-\psi^*g\|+\|f_j-\psi^*g\|\leq 2r
  \eeqa
  for $i,j=1,\ldots,n$.
  Hence, $B[f_i(k_1),r]\cap B[f_j(k_2),r]\neq\es$.
  Because $\|f_i\|\leq r+1$, $B[0,1]\cap B[f_i(k),r]\neq\es$ for all $i=1,\ldots,n$. Thus, the entire family of balls defining $\eta(y)$ has a pairwise non-empty intersection property. Moreover, owing to the collection of centers $\{0\}\cup\cup^n_{i=1}f_i(\psi^{-1}(y))$ being compact, $\eta(y)\neq\es$ is obtained.

 {\sc Claim:} $\eta$ is lower semicontinuous.
 
 Let $G\subseteq E$ be open, and $y_0\in\{y:\eta(y)\cap G\neq\es\}$ and $a\in\eta(y_0)\cap G$. Then, $\|a\|\leq1, f_i(\psi^{-1}(y_0))\subseteq B[a,r]$ and $B[a,\varepsilon]\subseteq G$ for certain $\varepsilon>0$ and for all $i=1,\ldots,n$.
 As $K$ is compact, the map $y\rightarrow\psi^{-1}(y)$ is upper semicontinuous. Hence, $N= \cap^n_{i=1}\{y:f_i(\psi^{-1}(y))\subseteq\textrm{int}B[a,r+\varepsilon]\}$ is an open set containing $y_0$. If $y\in N$, then $B[a,\varepsilon]\cap B[f_i(k),r]\neq\es$ for all $k\in \psi^{-1}(y)$ and for all $i=1,\ldots,n$. Additionally, $B[a,\e]\cap B[0,1]\neq\es$. Furthermore, because $E$ is a real Lindenstrauss space, $\eta(y)\cap B[a,\e]\neq\es$ for all $y\in N$. Thus, $N\subseteq\{y:\eta(y)\cap G\neq\es\}$, where $\{y:\eta(y)\cap G\neq\es\}$ is open and $\eta$ is lower semicontinuous.
 
 Now, applying Michael's selection theorem, a continuous selection $h: S\ra E$ is obtained, such that $h(y)\in\eta(y)$ for all $y\in S$. Accordingly, $\psi^*h\in \cap^n_{i=1}B[f_i,r]\cap B[0,1]\cap\psi^*C(S,E)$.
 This completes the proof.
  \end{proof}
  
 The following is obtained as a consequence of Theorem~\ref{1}.
  
  \begin{cor}
  Let $K, S, E, \psi$ be as in Theorem \ref{1}. Furthermore, $y_0\in S$ is set and let $M=\{\psi^*f:f\in C(S,E)~\textrm{and}~ f(y_0)=0\}$. Then, $(C(K,E),M,\mc{F}(C(K,E)))$ has strong property-$(R_1)$.
  \end{cor}
  \begin{proof}
  Let $f,r,\eta$ be similar to that in the proof of Theorem~\ref{1}. If $\psi^*g\in M\cap \cap^n_{i=1} B[f_i,r]$, $\|f_i(x)\|=\|f_i(x)-\psi^*g(x)\|\leq r$ for all $x\in \psi^{-1}(y_0)$ and for all $i=1,\ldots,n$. Hence, $0\in \eta(y_0)$. If we define $\eta_0:S \ra \mc{B}(E)$ by
  \begin{equation*}
  \eta_0(y)=\begin{cases}
  \eta (y) \quad &\text{if} \, y \neq y_0 \\
  \{0\} \quad &\text{if} \, y = y_0 \\
  \end{cases}
  \end{equation*}
  This $\eta_0$ is clearly lower semicontinuous. Subsequently, on applying Michael's selection theorem, a continuous selection $h:S\ra E$ is obtained, such that $h(y)\in\eta_0(y)$ for all $y\in K$. Hence, the assertion follows.
  \end{proof}

\section{Stability results}

For a Banach space $X$ we introduce the following notations.

$X_0=\bigoplus_{c_0}X=\{(x_n): x_n\in X, \lim_n \|x_n\|=0\}$, 

$X_\iy=\bigoplus_{\ell_\iy}X=\{(x_n):x_n\in X, \sup_n \|x_n\|<\iy\}$ 

$X_1=\bigoplus_{\ell_1}X=\{(x_n):x_n\in X, \sum_n \|x_n\|<\iy\}$ 

$\bigoplus_{\ell_1(m)} X=\oplus_{i=1}^m X$ with norm $\sum_{i=1}^m \|x_i\|$ and

$\bigoplus_{\ell_\iy(m)} X=\oplus_{i=1}^m X$ with norm $\max_{i=1}^m \|x_i\|$.

\begin{thm}
Let $Y$ be a subspace of $X$. Then, $(X,Y,\mc{F}(X))$ has (strong) property-$(R_1)$ if and only if $(X_0,Y_0, \mc{F}(X_0))$ has (strong) property-$(R_1)$.
\end{thm}
\begin{proof}
First, the result for property-$(R_1)$ is derived. 

It is sufficient to prove that $(X_0, Y_0, \mc{F}(X_0))$ exhibits property-$(R_1)$ if $(X,Y,\mc{F}(X))$ has property-$(R_1)$. Proof of this fact is outlined below.

Let $F\in \mc{F}(X_0)$ and $F(n)\subseteq X$ be the corresponding component, that is, $F(n)\in \mc{F}(X)$. Suppose $F=\{x_1,\ldots,x_k\}$, $y_0=(y_0(n))\in Y_0$ and $r_1,r_2>0$ be such that $S_{r_2}(F)\cap Y_0\neq \es$ and $r(y_0,F)<r_1+r_2$.

{\sc Claim:} $S_{r_2}(F)\cap B[y_0,r_1]\cap Y_0\neq\es$

It is clear that there exists $N$ sufficiently large, such that $0\in S_{r_2}(F(n))\cap B[y_0(n),r_1]\cap Y$, for all $n>N$.
Now, for $1\leq n\leq N$, we choose $y(n)\in S_{r_2}(F(n))\cap B[y_0(n),r_1]\cap Y$. As a result, an element $(y(n))$, a member of the set as specified in the claim above, exists. 

Nevertheless, regarding the remaining part, it is sufficient to prove that $(X_0, Y_0, \mc{F}(X_0))$ exhibits strong property-$(R_1)$ if $(X,Y,\mc{F}(X))$ exhibits strong property-$(R_1)$.

This study only shows that $\textrm{Cent}_{B_{Y_0}}(F)\neq\es$, for $F\in \mc{F}(X_0)$.

Let $F\in \mc{F}(X_0)$ and $F(n)\in \mc{F}(X)$ be as defined above. Based on the assumption, $\textrm{Cent}_{B_Y}(F(n))\neq\es$ for all $n\in\mathbb{N}$. Let $y(n)\in \textrm{Cent}_{B_Y}(F(n))$ be obtained for all $n\in\mathbb{N}$. 

It is clear that $y=(y(n))\in \bigoplus_{c_0}Y$, because $\textrm{rad}_{B_Y}(F(n))\ra 0$.  

Since for any $z=(z(1),z(2),\cdots)\in\bigoplus_{c_0}B_Y$, $r(y,F)\leq r(z,F)$ and because $\|(y(n))\|_\iy\leq 1$, $y\in\textrm{Cent}_{\bigoplus_{c_0}B_Y}(F)=\textrm{Cent}_{B_{Y_0}}(F)$, the result follows from Theorem~\ref{t01}.
\end{proof}

\begin{rem}
	\bla
\item	If $Y$ is a finite co-dimensional proximinal subspace of $c_0$, then there exists $n\in\mb{N}$, such that $Y=F\oplus_\iy Z$, where $F$ is a subspace of $\ell_\iy (n)$ and $Z=\{(x_i)\in c_0: x_i=0, 1\leq i\leq n\}$. From Theorem~\ref{T6}, it is clear that $Y$ has property-$(R_1)$ in $c_0$ if and only if $F$ has property-$(R_1)$ in $\ell_\iy (n)$. We do not know the characterization of $\al_i\in \ell_1 (n)$ where $1\leq i\leq m$, considering $dim (c_0/Y)=m$, satisfying the condition that $\cap_i \ker  (\al_i)$ has property-$(R_1)$ in $\ell_\iy (n)$. 
\item If $Y$ is a finite co-dimensional proximinal subspace of $c_0$, then from the decomposition $Y=F\oplus_\iy Z$, it is also clear that $Y$ has property-$(R_1)$ in $c_0$ if and only if $Y$ has property-$(R_1)$ in $\ell_\iy$.
\el
\end{rem}

We now consider the result for the $\ell_1$-sum. 

\begin{prop}\label{P3}
Let $X$ be a Banach space and $Y_1, Y_2$ be two subspaces of $X$. Then, for $F_1, F_2\in \mc{B}(X)$ the following can be obtained:
\bla
\item $r((y_1,y_2),F_1\times F_2)=r(y_1,F_1)+r(y_2,F_2)$ $\forall$ $(y_1,y_2)\in Y_1\oplus_1Y_2$.
\item $\emph{rad}_{Y_1\oplus_1Y_2}(F_1\times F_2)=\emph{rad}_{Y_1}(F_1)+\emph{rad}_{Y_2}(F_2)$.
\item $\emph{Cent}_{Y_1\oplus_1 Y_2}(F_1\times F_2)=\emph{Cent}_{Y_1}(F_1)\oplus_1\emph{Cent}_{Y_2}(F_2)$.
\item $d(0,\emph{Cent}_{Y_1\oplus_1Y_2}(F_1\times F_2))=d(0,\emph{Cent}_{Y_1}(F_1))+d(0,\emph{Cent}_{Y_2}(F_2))$.
\el
\end{prop}

Let $X$ be a Banach space. For a fixed $n$, define $\mc{H}=\{\Pi_{i=1}^n F_i:F_i\in\mc{F}(X)\}$. The following is obtained as a consequence of Proposition~\ref{P3}.

\begin{thm}\label{T8}
Let $X$ be a Banach space and $Y$ be a subspace of $X$.  Then, $(X,Y,\mc{F}(X))$ exhibits property-$(R_1)$ if and only if $(\bigoplus_{\ell_1(n)} X, \bigoplus_{\ell_1(n)} Y, \mc{H})$ exhibits property-$(R_1)$. 
\end{thm}

For a Banach space $X$, recall the notations defined before Section~1.2.
For $F\in \mc{F}(X)$, we identify $F\times F\times \ldots \times F (n-\mbox{times})$ with $\{(x_1,x_2,\ldots x_n,0,0\ldots):x_i\in F\}$. Let $\mathfrak{F}=\{\Pi_{i=1}^n F_i: F_i=F, F\in \mc{F}(X), n\in\mb{N}\}$. $\mathfrak{F}$ is now identified with a subfamily of $\mc{B}(X_1)$, more precisely a subfamily of $\mc{F}(X_1)$.

\begin{thm}\label{T7}
Let $Y$ be a subspace of $X$. Then, $(X,Y,\mc{F}(X))$ exhibits property-$(R_1)$ if and only if $(X_1,Y_1,\mathfrak{F})$ exhibits property-$(R_1)$.
\end{thm}
\begin{proof}
Here, proving that the condition is sufficient concludes the proof.

Let $\mc{W}\in \mathfrak{F}$, $y\in Y_1$, and $r_1,r_2>0$ be such that $r(y,\mc{W})<r_1+r_2$ and $S_{r_2}(\mc{W})\cap Y_1\neq\es$. Then, clearly $\mc{W}=\Pi_{i=1}^NF_i$ and we obtain a large $l\in \mb{N}(l>N)$ such that $\|y_i\|<r_1+r_2$, for all $i\geq l$. Let $\mc{W}_l=\{(x_1,x_2,\ldots,x_l):\exists~(w_i)\in\mc{W}, x_i=w_i, 1\leq i\leq l\}$ and $\La=(y_1,y_2,\ldots,y_l)$.

It is clear that $r(\La,\mc{W}_l)<r_1+r_2$ and $S_{r_2}(\mc{W}_l)\cap \bigoplus_{\ell_1(l)}Y\neq\es$.

Now, from Theorem~\ref{T8}, $S_{r_2}(\mc{W}_l)\cap B[\La,r_1]\cap \bigoplus_{\ell_1(l)}Y\neq\es$ is obtained. 

Let $(z_1,\cdots,z_l)$ be in the intersection above. Let $z=(z_1,\cdots,z_l,0,\ldots)$. Then, $z\in S_{r_2}(\mc{W})\cap B[y,r_1]\cap Y_1$. Hence, the conclusion follows.
\end{proof}

However, it is not known whether the $\ell_1$-sum remains stable for $(X_1,Y_1,\mc{F}(X_1))$. 

\section*{Acknowledgement}
The authors would like to send their gratitude to the referee for his / her careful reading of the manuscript and valuable comments.

\end{document}